\documentclass[12pt]{amsart}
\usepackage[utf8]{inputenc}
\usepackage{graphicx}
\usepackage{caption}
\usepackage{subcaption}
\usepackage{booktabs}
\usepackage{epstopdf}
\usepackage{inputenc}
\usepackage{tikz}
\usepackage{float}
\usepackage{enumitem}
\usepackage{amsmath,fullpage}
\usepackage{amssymb}
\usepackage{emptypage}
\usepackage{mathtools}
\usepackage[T1]{fontenc}
\usepackage[utf8]{inputenc}
\usepackage{lmodern}
\usepackage{geometry}
\usepackage{amsmath,amssymb,amsthm}
\usepackage{mathtools}
\usepackage{microtype}
\usepackage{enumitem}
\usepackage{hyperref}
\usepackage{xcolor}
\newcommand{\rev}[1]{#1}

\newtheorem*{theorem*}{Theorem}

\newtheorem{lemma}{Lemma}

\newtheorem*{acknowledgements*}{Acknowledgements}
\newtheorem{remark}{Remark}

\newenvironment{revision}
  {}
  {}

\makeatletter
\def\blfootnote{\gdef\@thefnmark{}\@footnotetext}
\makeatother

\def\house#1{\setbox1=\hbox{$\,#1\,$}%
\dimen1=\ht1 \advance\dimen1 by 2pt \dimen2=\dp1 \advance\dimen2 by 2pt
\setbox1=\hbox{\vrule height\dimen1 depth\dimen2\box1\vrule}%
\setbox1=\vbox{\hrule\box1}%
\advance\dimen1 by .4pt \ht1=\dimen1
\advance\dimen2 by .4pt \dp1=\dimen2 \box1\relax}

\begin{document}
	\title[Immersed Interface Method for Nonlinear Jump Conditions]
{An Immersed Interface Method for Parabolic Interface Problems
with Nonlinear Jump Conditions}

\author{So-Hsiang Chou}
\address{Department of Mathematics and Statistics,
Bowling Green State University,
Bowling Green, OH 43403, USA}
\email{chou@bgsu.edu}

\author{Patrick Nyadjo Fonga}
\address{Department of Mathematics and Statistics,
Bowling Green State University,
Bowling Green, OH 43403, USA}
\email{nyadjop@bgsu.edu}



\begin{abstract}
We develop an immersed interface finite difference method for a
one-dimensional nonlinear parabolic interface problem with jump condition
\[
[u]_\alpha=\lambda u^+u^-.
\]
The method combines a Crank--Nicolson immersed interface discretization
with an \(s\)-parameter reduction of the nonlinear interface condition,
thereby reducing the nonlinear coupling to a scalar quadratic equation.
We also discuss two different viewpoints for combining Newton iteration
with immersed interface discretization, namely the IIM--Newton and
Newton--IIM formulations.
Numerical experiments are presented to illustrate the behavior and
accuracy of the method.
\end{abstract}
\maketitle

\section{Introduction}

Immersed interface methods (IIM) provide an effective framework for
solving interface problems on unfitted grids while maintaining high-order
accuracy near interfaces.  In classical IIM formulations, the jump
conditions are assumed to be prescribed explicitly, so that correction
terms at irregular grid points can be constructed directly from the
known interface data.  For nonlinear interface conditions, however, the
jump quantities are not known a priori since they depend on the unknown
traces of the solution itself.  As a result, the implementation of
immersed interface corrections becomes substantially more delicate than
in the standard linear setting.

In this paper, we study immersed interface finite difference
approximations for the one-dimensional nonlinear parabolic interface
problem
({\bf NPIP}):
\begin{gather}
{\mathcal L}u:=u_t-(\beta u_x)_x + qu = f,
\qquad
x\in \Omega^-\cup\Omega^+,\quad t\in(0,T],
\label{eq:parabolic_prob1}
\\
u(a,t)=\xi(t),\qquad
u(b,t)=\eta(t),
\label{eq:parabolic_prob2}
\\
[u]_\alpha=\lambda u^+u^-,
\label{eq:parabolic_prob3}
\\
[\beta u_x]_\alpha=0,
\label{eq:parabolic_prob4}
\\
u(x,0)=g(x),
\qquad
x\in \Omega^-\cup\Omega^+.
\label{eq:parabolic_prob5}
\end{gather}
Here $\alpha$ is a fixed interface point,
$\Omega^-=(a,\alpha)$,
$\Omega^+=(\alpha,b)$,
and the diffusion coefficient $\beta=\beta(x)>0$ is piecewise constant:
\[
\beta=
\begin{cases}
\beta^-,
& x\in\Omega^-,
\\
\beta^+,
& x\in\Omega^+.
\end{cases}
\]
We use the standard notation
\[
[u]_\alpha=u^+-u^-,
\qquad
[\beta u_x]_\alpha
=
\beta^+u_x^+
-
\beta^-u_x^-,
\]
where
\[u^+ := \lim_{x\to \alpha^+} u(x,t),
    \qquad
    u^- := \lim_{x\to \alpha^-} u(x,t).\]

The nonlinear interface condition
\(
[u]_\alpha=\lambda u^+u^-
\)
couples the one-sided traces of the solution across the interface and
constitutes the principal nonlinearity of the problem.

Problems of this type arise in diffusion and transport models
involving interfaces, membranes, and heterogeneous media.
Numerical methods for interface problems with discontinuous
coefficients include fitted finite difference methods,
immersed interface methods, immersed finite element methods,
and related unfitted discretizations.
Among these, the immersed interface method introduced by
LeVeque and Li \cite{leveque1994immersed}
provides a systematic way to retain high-order finite difference
accuracy on Cartesian grids by incorporating correction terms
at irregular grid points near the interface.
Subsequent developments of immersed interface methods for
elliptic and parabolic interface problems may be found in
\cite{adams2002iimmg,li1998fem,li2006immersed},
while related immersed finite element approaches are discussed
in \cite{lin2013moving}.

Most immersed interface formulations are developed for interface
problems with prescribed jumps.
For nonlinear interface conditions, however, the interface jump
depends implicitly on the unknown solution traces, which introduces
additional difficulties in constructing immersed interface correction
terms.
A reduction of nonlinear interface conditions to a scalar equation
through an \(s\)-parameter formulation was proposed in
\cite{hetzer2007interface}.
Related fitted finite difference approaches for nonlinear jump
conditions were studied in \cite{Chou2025}.
The present work focuses instead on the immersed interface setting
and investigates how nonlinear jump conditions may be incorporated
into a classical immersed interface discretization.

To handle the nonlinear jump relation, we employ the
\(s\)-parameter reduction introduced in
\cite{hetzer2007interface}. The solution is represented as
\begin{equation}\label{eq:s1}
u=u_0+s u_1.
\end{equation}
Here \(u_0\) satisfies homogeneous interface jumps and \(u_1\)
corresponds to a unit solution jump. The nonlinear interface condition
then reduces to a scalar quadratic equation for the parameter \(s\).
This decomposition separates the nonlinear interface coupling from the
linear bulk equations and allows the immersed interface discretization to
be applied to two auxiliary linear problems.

The purpose of the present paper is not to develop a new immersed
interface stencil for general nonlinear interface problems, but rather
to investigate how nonlinear jump conditions can be incorporated into a
classical immersed interface framework.  In particular, we examine two
different viewpoints for combining Newton iteration with immersed
interface discretization.

The first viewpoint, referred to here as the IIM--Newton approach,
applies the immersed interface discretization first and then performs
Newton iteration on the resulting nonlinear algebraic system.  In this
case, the nonlinearity is localized near the irregular interface grid
points, and the Jacobian corrections remain sparse.

The second viewpoint, referred to as the Newton--IIM approach,
linearizes the nonlinear interface problem first at the continuous
level and then attempts to discretize the resulting correction equation
by the immersed interface method.  We also discuss several practical difficulties associated with
this latter strategy within the standard immersed interface framework.

For the numerical verification tests presented in this paper, the
one-sided interface traces entering the scalar equation for \(s\) are
evaluated from the manufactured exact solution.  This allows the
experiments to isolate the immersed interface discretization and the
nonlinear scalar reduction from additional trace-reconstruction errors.
\begin{revision}
The separate issue of recovering interface traces from unfitted
numerical data is discussed briefly below.  A systematic study of
trace-recovery procedures, including reconstructions that incorporate
the interface jump relations, is the subject of separate work.
\end{revision}

\rev{To expose the nonlinear algebraic structure without the additional
notation associated with time discretization, we first develop the
IIM--Newton formulation for the corresponding stationary elliptic
interface problem.  This problem also serves as a model for the
nonlinear algebraic problem arising at each time level after an
implicit discretization of the parabolic equation.  The parabolic
problem is then treated by combining the Crank--Nicolson immersed
interface discretization with the $s$-parameter reduction.}

The remainder of the paper is organized as follows.
Section~2 derives the immersed interface finite difference approximation
for a related linear interface problem and records the correction terms
at irregular grid points.
\rev{Section~3 presents the IIM--Newton formulation for the stationary
nonlinear elliptic model problem.}
Section~4 contains numerical experiments illustrating the behavior of
the method.
Finally, Section~5 discusses the Newton--IIM formulation and explains
its practical difficulties within the immersed interface framework.

\section{Immersed Interface Discretization for a Linear Problem}
\label{sec:linear-iim}

The purpose of this section is to derive the immersed interface finite
difference approximation for a linear parabolic interface problem with
prescribed jumps.  This linear problem serves as the building block for
the nonlinear problem \({\bf NPIP}\).

Since the correction terms at irregular grid points depend sensitively
on the convention used for the discrete operator, we briefly record the
derivation employed in the present implementation.

We consider the linear parabolic interface problem
\begin{gather}
u_t=(\beta u_x)_x-qu+f,
\qquad
x\in\Omega^-\cup\Omega^+,
\quad
t\in(0,T],
\label{eq:linear_parabolic1}
\\
u(a,t)=\xi(t),
\qquad
u(b,t)=\eta(t),
\label{eq:linear_parabolic2}
\\
[u]_\alpha
=
(1-A(t))u^+
+
(B(t)-1)u^-
+
w(t),
\label{eq:linear_parabolic3}
\\
[\beta u_x]_\alpha
=
\beta^+u_x^+
-
\beta^-u_x^-
=
v(t),
\label{eq:linear_parabolic4}
\\
u(x,0)=g(x),
\qquad
x\in\Omega^-\cup\Omega^+.
\label{eq:linear_parabolic5}
\end{gather}
Equivalently, the jump condition may be written as
\[
A(t)u^+ - B(t)u^- = w(t).
\]
When \(A(t)=B(t)=1\), this reduces to the standard prescribed jump
condition
\[
[u]_\alpha=w(t).
\]

\begin{revision}
We assume that $A,B,w\in C^1([0,T])$, that $v$ is continuous
on $[0,T]$, that
\[
A(t)B(t)\neq0,
\qquad 0\le t\le T,
\]
and that the one-sided limits
\[
f^\pm(t)=\lim_{x\to\alpha^\pm}f(x,t)
\]
exist.  The initial data are assumed to be compatible with the
boundary and interface conditions at $t=0$.
\end{revision}

Let
\[
a=x_0<x_1<\cdots<x_N=b
\]
be a uniform mesh with mesh size \(h\), and assume that the interface
point satisfies
\[
x_j\le \alpha < x_{j+1}.
\]
The grid points \(x_j\) and \(x_{j+1}\) are the irregular points.

For the time discretization, we use the $\theta$-method:
\begin{align}
\frac{U_i^{n+1}-U_i^n}{\Delta t}
&=
\theta
\left(
\gamma_{i,1}U_{i-1}^{n+1}
+
\gamma_{i,2}U_i^{n+1}
+
\gamma_{i,3}U_{i+1}^{n+1}
+
C_i^{n+1}
-
(qU)_i^{n+1}
\right)
\nonumber
\\
&\quad
+
(1-\theta)
\left(
\gamma_{i,1}U_{i-1}^{n}
+
\gamma_{i,2}U_i^{n}
+
\gamma_{i,3}U_{i+1}^{n}
+
C_i^{n}
-
(qU)_i^{n}
\right)
\nonumber
\\
&\quad
+
f(x_i,t_{n+\theta}).
\label{eq:theta_scheme}
\end{align}
Here the coefficients $\gamma_{i,k}$ and the correction term $C_i^n$
depend on whether $x_i$ is a regular or an irregular grid point.

Away from the interface, the standard centered approximation is used.
Thus, for $i\notin\{j,j+1\}$,
\[
\gamma_{i,1}=\frac{\beta_{i-\frac12}}{h^2},
\qquad
\gamma_{i,2}
=
-\frac{\beta_{i-\frac12}+\beta_{i+\frac12}}{h^2},
\qquad
\gamma_{i,3}
=
\frac{\beta_{i+\frac12}}{h^2},
\qquad
C_i^n=0.
\]
Here \(\theta=1\) corresponds to backward Euler, while
\(\theta=1/2\) gives the Crank--Nicolson method.

\begin{revision}
In the fully discrete scheme, the correction terms
$C_i^n$ and $C_i^{n+1}$ are evaluated from the interface data at
$t_n$ and $t_{n+1}$, respectively.  Thus, for the Crank--Nicolson
scheme ($\theta=1/2$), the two correction terms enter through their
trapezoidal average.  The quantities $A_t$, $B_t$, and $w_t$
appearing in the correction formulas below are evaluated at the
corresponding time levels.  When these functions are prescribed
analytically, as in the verification problems considered here, their
derivatives are evaluated directly.
\end{revision}

It remains to determine the immersed interface correction terms at the
irregular grid points.  We derive them by Taylor expansion about the
interface.

\subsection*{Correction term at \(x_j\)}

At the point \(x_j\), we expand
\(u_{j-1}\),
\(u_j\), and
\(u_{j+1}\)
about \(\alpha^-\):
\begin{align}
u_{j-1}
&=
u^-
+
u_x^-(x_{j-1}-\alpha)
+
\frac12u_{xx}^-(x_{j-1}-\alpha)^2
+
\mathcal O(h^3),
\\
u_j
&=
u^-
+
u_x^-(x_j-\alpha)
+
\frac12u_{xx}^-(x_j-\alpha)^2
+
\mathcal O(h^3),
\\
u_{j+1}
&=
u^+
+
u_x^+(x_{j+1}-\alpha)
+
\frac12u_{xx}^+(x_{j+1}-\alpha)^2
+
\mathcal O(h^3).
\end{align}

The interface conditions imply
\[
u^+
=
\frac{Bu^-+w}{A},
\qquad
u_x^+
=
\frac{v+\beta^-u_x^-}{\beta^+}.
\]

Using the equation on both sides of the interface together with the
time derivative of
\[
Au^+-Bu^-=w,
\]
we obtain
\[
u_{xx}^+
=
\frac{B\beta^-}{A\beta^+}u_{xx}^-
+
\left(
\frac{B(q^+-q^-)}{A\beta^+}
+
\frac1{\beta^+}\left(\frac{B}{A}\right)_t
\right)u^-
+
\frac{q^+w-(Af^+-Bf^-)+w_t}{A\beta^+}
-
\frac{A_tw}{\beta^+A^2}.
\]

Matching the coefficients of
\(u^-\),
\(u_x^-\), and
\(u_{xx}^-\)
gives the correction term
\begin{equation}
\label{eq:Cj_corrected}
C_j
=
-\gamma_{j,3}
\left[
\frac{w}{A}
+
\frac{v}{\beta^+}(x_{j+1}-\alpha)
+
\frac{(x_{j+1}-\alpha)^2}{2}
\left(
\frac{q^+w-(Af^+-Bf^-)+w_t}{A\beta^+}
-
\frac{A_tw}{\beta^+A^2}
\right)
\right].
\end{equation}

\subsection*{Correction term at \(x_{j+1}\)}

We next derive the correction term at \(x_{j+1}\).  Expanding about
\(\alpha^+\) gives
\begin{align}
u_j
&=
u^-
+
u_x^-(x_j-\alpha)
+
\frac12u_{xx}^-(x_j-\alpha)^2
+
\mathcal O(h^3),
\\
u_{j+1}
&=
u^+
+
u_x^+(x_{j+1}-\alpha)
+
\frac12u_{xx}^+(x_{j+1}-\alpha)^2
+
\mathcal O(h^3),
\\
u_{j+2}
&=
u^+
+
u_x^+(x_{j+2}-\alpha)
+
\frac12u_{xx}^+(x_{j+2}-\alpha)^2
+
\mathcal O(h^3).
\end{align}

The interface conditions now imply
\[
u^-
=
\frac{Au^+-w}{B},
\qquad
u_x^-
=
\frac{\beta^+u_x^+-v}{\beta^-}.
\]

Proceeding as above yields
\[
u_{xx}^-
=
\frac{A\beta^+}{B\beta^-}u_{xx}^+
+
\left(
-\frac{A(q^+-q^-)}{B\beta^-}
+
\frac1{\beta^-}\left(\frac{A}{B}\right)_t
\right)u^+
+
\frac{
Af^+-Bf^-
-q^-w
-w_t
+\frac{B_t}{B}w
}{B\beta^-}.
\]

Matching the coefficients of
\(u^+\),
\(u_x^+\), and
\(u_{xx}^+\)
gives the correction term
\begin{equation}
\label{eq:Cjp1_corrected}
C_{j+1}
=
\gamma_{j+1,1}
\left[
\frac{w}{B}
+
\frac{v}{\beta^-}(x_j-\alpha)
+
\frac{(x_j-\alpha)^2}{2}
\left(
\frac{
q^-w
-(Af^+-Bf^-)
+w_t
-\frac{B_t}{B}w
}
{B\beta^-}
\right)
\right].
\end{equation}

When \(A=B=1\), the jump condition reduces to
\[
[u]_\alpha=w(t),
\]
and the correction terms simplify to
\[
C_j
=
-\gamma_{j,3}
\left[
w
+
\frac{v}{\beta^+}(x_{j+1}-\alpha)
+
\frac{(x_{j+1}-\alpha)^2}{2}
\frac{q^+w-[f]+w_t}{\beta^+}
\right],
\]
and
\[
C_{j+1}
=
\gamma_{j+1,1}
\left[
w
+
\frac{v}{\beta^-}(x_j-\alpha)
+
\frac{(x_j-\alpha)^2}{2}
\frac{q^-w-[f]+w_t}{\beta^-}
\right].
\]

\begin{revision}
\subsection{Spatial consistency and accuracy}
\label{subsec:spatial-consistency}

We briefly record the spatial consistency of the preceding immersed
interface discretization and explain why the lower local order at the
two irregular grid points is compatible with second-order global
accuracy.  For simplicity, the discussion is given for the
homogeneous forcing case $f=0$; the same local argument applies when
the prescribed forcing is retained.

\begin{lemma}
\label{lem:spatial-consistency}
Fix $t\in[0,T]$, and suppose that the one-sided solution satisfies
\[
u^\pm(\cdot,t)\in C^4(\overline{\Omega^\pm}).
\]
Assume also that the interface data are sufficiently smooth for the
preceding Taylor expansions.  Let $\tau_i$ denote the spatial
truncation error obtained by inserting the exact nodal values into
the immersed interface difference operator, including the correction
terms \eqref{eq:Cj_corrected} and \eqref{eq:Cjp1_corrected}.  Then
\[
\tau_i=\mathcal O(h^2),
\qquad
i\notin\{j,j+1\},
\]
while at the two irregular grid points,
\[
\tau_j=\mathcal O(h),
\qquad
\tau_{j+1}=\mathcal O(h).
\]
\end{lemma}

\begin{proof}
At a regular grid point the stencil is the standard centered
second-order approximation, and hence
\[
\tau_i=\mathcal O(h^2).
\]

At an irregular grid point, the Taylor expansions about the interface
are matched through the quadratic terms by the correction formulas
\eqref{eq:Cj_corrected} and \eqref{eq:Cjp1_corrected}.  The first
unmatched terms in the expansions are therefore of order
$\mathcal O(h^3)$.  Since the coefficients of the second-difference
operator are of order $\mathcal O(h^{-2})$, the remaining local
residual is of order
\[
\mathcal O(h^{-2})\,\mathcal O(h^3)
=
\mathcal O(h).
\]
The same argument applies at both irregular grid points.
\end{proof}

\begin{remark}
\label{rem:global-order}
The first-order truncation errors at the two irregular grid points do
not by themselves imply a loss of global second-order accuracy.  This
is a familiar feature of one-dimensional immersed interface
discretizations.  The irregular defects are supported at only two
neighboring grid points, whereas the discrete inverse associated with
the second-order elliptic difference operator has entries of order
$h$ in the corresponding discrete Green-function representation.
Hence an $\mathcal O(h)$ defect supported at a fixed number of grid
points produces an $\mathcal O(h^2)$ contribution to the solution
error.  The $\mathcal O(h^2)$ truncation errors at the regular grid
points likewise give an $\mathcal O(h^2)$ global contribution.
Thus the spatial discretization is consistent with second-order
global accuracy despite the lower local order at the two irregular
points.  The numerical results in Section~4 confirm this behavior.
\end{remark}
\end{revision}
\section{IIM--Newton Formulation: A Stationary Model Problem}
\label{sec:newton-elliptic}

\begin{revision}
To make the nonlinear algebraic structure transparent, we first
consider the stationary analogue of the parabolic interface problem.
The same structure arises at each time level after an implicit
discretization of the parabolic equation.  The essential issue is
that, unlike the prescribed-jump problem of Section~2, the solution
jump now depends on the unknown interface traces.  We show that this
nonlinearity can be incorporated directly into the immersed interface
correction terms, so that the nonlinear contribution remains confined
to the irregular grid points.
\end{revision}

We consider the elliptic interface problem
\begin{gather}
-(\beta u')'+qu=f,
\qquad
x\in\Omega^-\cup\Omega^+,
\label{eq:elliptic_newton1}
\\
u(a)=\xi,
\qquad
u(b)=\eta,
\label{eq:elliptic_newton2}
\\
[u]_\alpha=\lambda u^+u^-,
\label{eq:elliptic_newton3}
\\
[\beta u']_\alpha=v.
\label{eq:elliptic_newton4}
\end{gather}

The nonlinear jump condition may be rewritten as
\begin{equation}
\label{eq:jump_w_left_trace}
[u]_\alpha
=
w
=
\frac{\lambda (u^-)^2}{1-\lambda u^-}.
\end{equation}

\begin{revision}
Thus the nonlinear interface condition has the same algebraic form as
a prescribed solution jump, except that the jump value is now an
unknown function of the one-sided trace $u^-(\alpha)$.  This
observation allows the linear immersed interface construction of
Section~2 to be retained: the prescribed quantity $w$ in the
irregular-point correction terms is replaced by a nonlinear function
of the discrete solution.
\end{revision}

Let
\[
a=x_0<x_1<\cdots<x_N<x_{N+1}=b
\]
be a uniform grid with mesh size \(h\), and assume that
\[
x_I\le \alpha <x_{I+1}.
\]
The grid points \(x_I\) and \(x_{I+1}\) are the irregular points.

For a linear interface problem with prescribed jumps, the immersed
interface discretization has the form
\begin{equation}
\label{eq:linear_iim_matrix}
A_hU=F_h+C_h,
\end{equation}
where \(U=(U_1,\ldots,U_N)^T\),
\(A_h\) contains the standard finite difference coefficients together
with the modified irregular-point coefficients, and \(C_h\) contains
the immersed interface correction terms.  The vector \(C_h\) is zero
except possibly in the two entries corresponding to the irregular
grid points.

For the nonlinear problem, the correction terms depend on the
nonlinear jump \(w\).  Consequently, the discrete system becomes
\begin{equation}
\label{eq:nonlinear_iim_system}
A_hU=F_h+\Phi(U),
\end{equation}
where
\[
\Phi(U)
=
\begin{bmatrix}
0\\
\vdots\\
\Phi_I(U)\\
\Phi_{I+1}(U)\\
\vdots\\
0
\end{bmatrix}.
\]
Thus the nonlinearity appears only in the correction terms at the
irregular points.

\begin{revision}
To express the nonlinear jump in terms of the discrete unknowns, an
approximation of the one-sided trace $u^-(\alpha)$ is required.  In
the computations below, we use the three neighboring grid values on
the minus side and define the quadratic Lagrange extrapolation
\[
\widetilde U^-
=
U_{I-2}\ell_{I-2}(\alpha)
+
U_{I-1}\ell_{I-1}(\alpha)
+
U_I\ell_I(\alpha),
\]
where
\[
\ell_{I-2}(x)
=
\frac{(x-x_{I-1})(x-x_I)}
     {(x_{I-2}-x_{I-1})(x_{I-2}-x_I)},
\]
\[
\ell_{I-1}(x)
=
\frac{(x-x_{I-2})(x-x_I)}
     {(x_{I-1}-x_{I-2})(x_{I-1}-x_I)},
\]
and
\[
\ell_I(x)
=
\frac{(x-x_{I-2})(x-x_{I-1})}
     {(x_I-x_{I-2})(x_I-x_{I-1})}.
\]
For a smooth one-sided solution, this is the usual quadratic
same-side extrapolation to the interface.

Replacing $u^-$ in \eqref{eq:jump_w_left_trace} by
$\widetilde U^-$, we define the discrete nonlinear jump
\begin{equation}
\label{eq:mu_definition}
\mu(U)
=
\frac{\lambda(\widetilde U^-)^2}
     {1-\lambda\widetilde U^-}.
\end{equation}

The immersed interface correction terms are then obtained by replacing
the prescribed jump $w$ in the corresponding linear correction
formulas by $\mu(U)$.  Thus
\[
\Phi_I(U)=C_I(\mu(U)),
\qquad
\Phi_{I+1}(U)=C_{I+1}(\mu(U)).
\]

This replacement is the essential step in the IIM--Newton
formulation.  The finite-difference operator $A_h$ is unchanged from
the corresponding linear immersed interface problem; all dependence
on the nonlinear jump law is collected in the correction vector
$\Phi(U)$.  Moreover,
\[
\Phi_i(U)=0,
\qquad
i\notin\{I,I+1\},
\]
so that the nonlinear correction has fixed local support, independent
of the total number of grid points.
\end{revision}

\begin{revision}
For the stationary interface problem considered in this section, the
correction terms depend linearly on the prescribed solution jump $w$
when the remaining interface data are fixed.  We may therefore write
\[
C_I(w)=\kappa_I w+r_I,
\qquad
C_{I+1}(w)=\kappa_{I+1}w+r_{I+1},
\]
where $\kappa_I$, $\kappa_{I+1}$ and the terms $r_I,r_{I+1}$ are
determined by the corresponding linear immersed interface
discretization.  Replacing $w$ by $\mu(U)$ gives
\[
C_I(\mu(U))=\kappa_I\mu(U)+r_I,
\qquad
C_{I+1}(\mu(U))=\kappa_{I+1}\mu(U)+r_{I+1}.
\]
Thus the nonlinear dependence enters through the scalar quantity
$\mu(U)$.

If the solution jump is prescribed independently of $U$, then
$\mu(U)$ is replaced by that prescribed value and $\Phi'(U)=0$;
the formulation therefore reduces immediately to the linear immersed
interface scheme of Section~2.
\end{revision}

Define the nonlinear residual
\begin{equation}
\label{eq:newton_residual}
G(U)
=
A_hU-F_h-\Phi(U).
\end{equation}
The nonlinear algebraic problem
\eqref{eq:nonlinear_iim_system}
is therefore equivalent to
\[
G(U)=0.
\]

Newton's method is given by the iteration
\begin{equation}
\label{eq:newton_linear_system}
G'(U^{(k)})\delta U^{(k)}
=
-G(U^{(k)}),
\end{equation}
followed by the update
\begin{equation}
\label{eq:newton_update}
U^{(k+1)}
=
U^{(k)}
+
\delta U^{(k)}.
\end{equation}

Since
\[
G'(U)=A_h-\Phi'(U),
\]
\begin{revision}
the Jacobian differs from the linear immersed interface matrix only
through the localized derivative $\Phi'(U)$.  Since
$\widetilde U^-$ depends only on $U_{I-2}$, $U_{I-1}$, and $U_I$,
the Jacobian correction contains only a few nonzero entries.  Thus
the sparse bulk structure of the linear immersed interface
discretization is preserved, while the nonlinear modification remains
localized near the interface.

The derivative of $\mu(U)$ with respect to $\widetilde U^-$ is
\[
\frac{d\mu}{d\widetilde U^-}
=
\frac{
2\lambda\widetilde U^-
-\lambda^2(\widetilde U^-)^2
}{
(1-\lambda\widetilde U^-)^2
}.
\]
Consequently,
\begin{align}
\frac{\partial\mu}{\partial U_{I-2}}
&=
\frac{
2\lambda\widetilde U^-
-\lambda^2(\widetilde U^-)^2
}{
(1-\lambda\widetilde U^-)^2
}
\,\ell_{I-2}(\alpha),
\\
\frac{\partial\mu}{\partial U_{I-1}}
&=
\frac{
2\lambda\widetilde U^-
-\lambda^2(\widetilde U^-)^2
}{
(1-\lambda\widetilde U^-)^2
}
\,\ell_{I-1}(\alpha),
\\
\frac{\partial\mu}{\partial U_I}
&=
\frac{
2\lambda\widetilde U^-
-\lambda^2(\widetilde U^-)^2
}{
(1-\lambda\widetilde U^-)^2
}
\,\ell_I(\alpha).
\end{align}
\end{revision}

\begin{remark}
The quantity
\[
1-\lambda\widetilde U^-
\]
must remain nonzero during the Newton iteration.  If the denominator
becomes too small, the Newton step should be damped or restarted from
a different initial guess.  This illustrates one practical difficulty
of directly incorporating nonlinear jumps into an immersed interface
Newton iteration.
\end{remark}

\begin{revision}
The quadratic extrapolation above provides a simple way to close the
IIM--Newton algebraic system using computed grid values and is the
trace approximation used in the elliptic experiment of
Section~4.1.  More generally, one-sided interface traces may be
recovered by higher-order same-side reconstruction or by an
IIM-consistent reconstruction incorporating the jump relations.
A systematic analysis of such trace-recovery procedures is beyond
the scope of the present paper.

In the manufactured-solution verification of the $s$-parameter
method in Section~4.2, the required auxiliary interface traces are
instead evaluated from the exact auxiliary solutions.  This choice
is made deliberately to isolate the accuracy of the immersed
interface discretization and the nonlinear scalar reduction from
additional trace-reconstruction errors.
\end{revision}

\section{Numerical Experiments}
\label{sec:numerical-experiments}

In this section, we present numerical experiments illustrating the
behavior of the immersed interface methods discussed in the previous
sections.

The first example concerns a nonlinear elliptic interface problem and is
used to illustrate the behavior of the IIM--Newton iteration for a
problem possessing multiple solutions.  In particular, we demonstrate
the sensitivity of the iteration to the choice of the initial iterate.

The second example concerns a nonlinear parabolic interface problem and
illustrates the Crank--Nicolson immersed interface scheme combined with
the $s$-parameter reduction.  In this verification example, the
one-sided interface traces entering the quadratic equation for the
parameter $s$ are evaluated from the exact auxiliary solutions.  This
allows us to isolate the immersed interface discretization and the
scalar reduction from additional trace-reconstruction errors.

Throughout the experiments, the error is measured in the discrete
maximum norm
\[
\|E_h\|_\infty
=
\max_i |U_i-u(x_i)|,
\]
where $U_i$ denotes the numerical approximation and $u(x_i)$ denotes
the exact solution at the grid point $x_i$.

\subsection{A nonlinear elliptic interface problem}
\label{subsec:elliptic-numerical-example}

We first consider the nonlinear elliptic interface problem on
$(-1,1)$ with interface point $\alpha=0$:
\begin{align}
-(\beta u')'
&=
\sin(\pi x),
\qquad
-1<x<0
\quad \text{or} \quad
0<x<1,
\label{eq:elliptic_ex1_pde}
\\
[u]
&=
\lambda u^+u^-,
\label{eq:elliptic_ex1_jump_u}
\\
[\beta u']
&=
0,
\label{eq:elliptic_ex1_jump_flux}
\\
u(-1)
&=
u(1)
=
0.
\label{eq:elliptic_ex1_bc}
\end{align}

\begin{revision}
The parameters are
\[
\lambda=1.5,
\qquad
\beta^-=2,
\qquad
\beta^+=1.
\]
\end{revision}

The exact solution has the form
\begin{equation}
u(x)=
\begin{cases}
\dfrac{1}{\pi^2\beta^-}\sin(\pi x)
+
A(x+1),
& -1\le x<0,
\\[2mm]
\dfrac{1}{\pi^2\beta^+}\sin(\pi x)
+
B(x-1),
& 0<x\le 1.
\end{cases}
\label{eq:elliptic_ex1_exact_form}
\end{equation}
Let
\[
r=\frac{\beta^-}{\beta^+}.
\]
The flux continuity condition gives $B=rA$, and substitution into the
nonlinear jump condition yields
\[
\lambda rA^2=(r+1)A.
\]
Consequently, the problem possesses two exact solutions,
\[
A=0,
\qquad
A=\frac{r+1}{\lambda r}.
\]
The first is a continuous solution, denoted by $u_c$, while the second
is a discontinuous solution, denoted by $u_d$.

Because the discrete nonlinear system also possesses multiple solutions,
the convergence of Newton's method depends on the initial iterate.  In
our computations, the basin of attraction of the continuous solution
$u_c$ appears to be larger.  For example, initial iterates of the form
\[
t\,u_d(x_i),
\qquad
0\le t\le 0.67,
\]
converge to the continuous solution $u_c$, while larger values of $t$
may lead to convergence toward the discontinuous solution $u_d$.

Figures~\ref{fig:elliptic-discontinuous-solution} and
\ref{fig:elliptic-continuous-solution} show the two exact solutions
together with their immersed interface approximations.

\begin{figure}[H]
\centering
\includegraphics[width=3.5in]{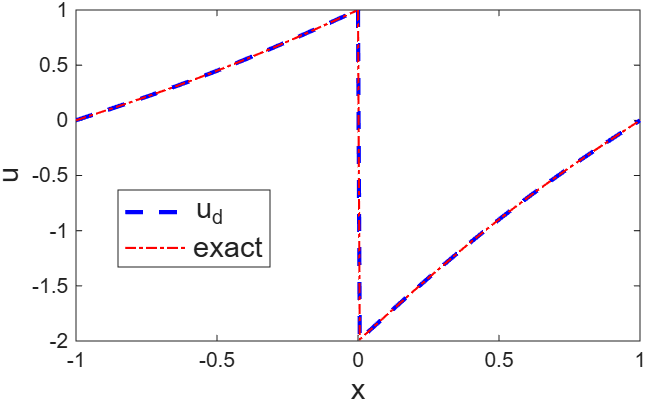}
\caption{Discontinuous solution $u_d$ and its immersed interface
approximation.}
\label{fig:elliptic-discontinuous-solution}
\end{figure}

\begin{figure}[H]
\centering
\includegraphics[width=3.5in]{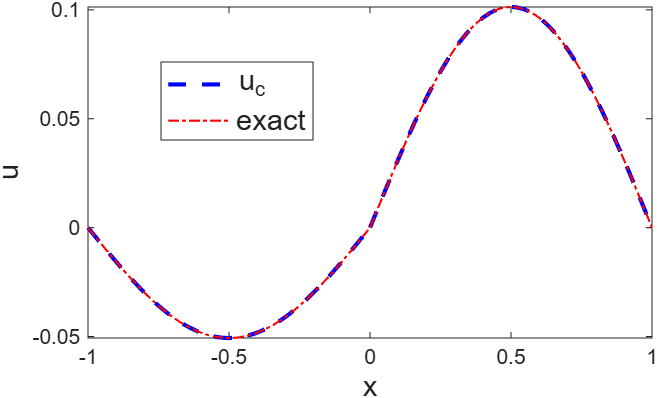}
\caption{Continuous solution $u_c$ and its immersed interface
approximation.}
\label{fig:elliptic-continuous-solution}
\end{figure}

\begin{revision}
Since the mesh size is halved successively, the observed order is
computed by
\[
p_h=\frac{\log(E_h/E_{h/2})}{\log 2}.
\]
Table~\ref{tab:elliptic_ex1_error} reports the $L^\infty$-errors
for the discontinuous solution.  The results show second-order
convergence.
\end{revision}

\begin{table}[H]
\centering
\caption{$L^\infty$-errors and observed orders for the nonlinear
elliptic example.}
\label{tab:elliptic_ex1_error}
\vskip 0.1in
\begin{tabular}{|c|c|c|}
\hline
$h$ & $L^\infty$-error & Observed order \\
\hline
0.05     & $2.0859\times 10^{-4}$ & -- \\
\hline
0.025    & $5.2099\times 10^{-5}$ & 2.001 \\
\hline
0.0125   & $1.3022\times 10^{-5}$ & 2.001 \\
\hline
0.00625  & $3.2553\times 10^{-6}$ & 2.000 \\
\hline
0.003125 & $8.1381\times 10^{-7}$ & 2.000 \\
\hline
\end{tabular}
\end{table}

\begin{revision}
\begin{revision}
To examine the behavior for a substantially larger diffusion contrast,
we repeated the computation, keeping $\lambda=1.5$, with
\[
\beta^-=1,\qquad \beta^+=250,
\]
so that $\beta^+/\beta^-=250$.
\end{revision}The errors and observed orders are
reported in Table~\ref{tab:elliptic_high_contrast}.  The method again
exhibits second-order convergence in the maximum norm.
\end{revision}

\begin{table}[H]
\centering
\caption{$L^\infty$-errors and observed orders for the nonlinear
elliptic problem with $\beta^+/\beta^-=250$.}
\label{tab:elliptic_high_contrast}
\vskip 0.1in
\begin{tabular}{|c|c|c|}
\hline
$h$ & $L^\infty$-error & Observed order \\
\hline
0.05    & $2.0859\times10^{-4}$ & -- \\
\hline
0.025   & $5.2099\times10^{-5}$ & 2.001 \\
\hline
0.0125  & $1.3022\times10^{-5}$ & 2.000 \\
\hline
0.00625 & $3.2556\times10^{-6}$ & 2.000 \\
\hline
\end{tabular}
\end{table}

\subsection{A nonlinear parabolic interface problem}
\label{subsec:parabolic-numerical-example}

We next consider the nonlinear parabolic interface problem
\begin{align}
u_t-(\beta u_x)_x
&=
f(x,t),
\qquad
-1<x<0
\quad \text{or} \quad
0<x<1,
\qquad
t\in(0,T],
\label{eq:parabolic_ex_pde}
\\
u(-1,t)
&=
0,
\qquad
u(1,t)
=
\frac1{1-\lambda},
\label{eq:parabolic_ex_bc}
\\
[u]
&=
\lambda u^+u^-,
\label{eq:parabolic_ex_jump_u}
\\
[\beta u_x]
&=
0.
\label{eq:parabolic_ex_jump_flux}
\end{align}

The parameters in the first set of computations are
\[
\lambda=-2.5,
\qquad
\beta^-=1,
\qquad
\beta^+=0.1.
\]
The source term is chosen so that the exact solution is
\begin{equation}
u(x,t)=
\begin{cases}
e^{-t}\sin(\pi x)+x+1,
& -1\le x<0,
\\[2mm]
\dfrac1{1-\lambda}
+
\dfrac{\beta^-}{\beta^+}
(\pi e^{-t}+1)(x-x^2),
& 0<x\le 1.
\end{cases}
\label{eq:parabolic_exact_solution}
\end{equation}

The solution is represented by the $s$-parameter decomposition
\[
u(x,t)=u_0(x,t)+s(t)u_1(x,t),
\]
where $u_0$ satisfies the linear auxiliary problem with zero solution
jump, while $u_1$ satisfies the corresponding homogeneous problem
with unit solution jump and zero flux jump.  The scalar parameter
$s(t)$ is determined from the nonlinear jump condition.

At each time level, $s(t)$ satisfies
\[
s(t)
=
\lambda
\bigl(u_0^+(t)+s(t)u_1^+(t)\bigr)
\bigl(u_0^-(t)+s(t)u_1^-(t)\bigr).
\]
Equivalently,
\[
a_2s^2+a_1s+a_0=0,
\]
where
\[
a_2=\lambda u_1^+u_1^-,
\qquad
a_1=
\lambda(u_0^-u_1^+ + u_0^+u_1^-)-1,
\qquad
a_0=\lambda u_0^+u_0^-.
\]

\begin{revision}
The scalar quadratic equation may admit two distinct real roots, a
repeated root, or no real root.  As discussed in our earlier work
\cite{Chou2025}, the nonlinear interface condition can therefore give
rise to multiple solution branches.  In a time-dependent computation,
a natural way to follow a prescribed branch is by continuation.  The
initial root is selected consistently with the initial data, and at
each subsequent time level the admissible real root closest to the
value of $s$ at the preceding time level is selected.  This criterion
avoids unintended switching between solution branches.

In the manufactured-solution experiment reported here, the desired
solution branch is known explicitly, and the root corresponding to
that branch is selected.  If the discriminant vanishes, the two roots
coincide.  If the discriminant becomes negative beyond roundoff
tolerance, the discrete scalar equation has no real root and the
continuation procedure must be reconsidered.
\end{revision}

\begin{revision}
In the present manufactured-solution experiment, the one-sided
auxiliary traces entering the scalar equation for $s$ are evaluated
from the exact auxiliary solutions.  This choice is deliberate: it
allows us to isolate the accuracy of the immersed interface
discretization and the nonlinear scalar reduction from additional
trace-reconstruction errors.  In a general computation, the required
one-sided traces can instead be recovered from nearby numerical
values, for example by one-sided extrapolation or by an
IIM-consistent reconstruction incorporating the interface jump
relations.  A detailed study of numerical trace recovery is beyond
the scope of the present paper.
\end{revision}

\begin{revision}
For the computations reported below, the Crank--Nicolson method
($\theta=1/2$) is used.  For each spatial mesh, the number of time
steps is chosen as
\[
M=\operatorname{round}(T/h),
\qquad
\Delta t=\frac{T}{M}.
\]
Thus $\Delta t$ is comparable to $h$, while the final time is reached
exactly, $t_M=T$.  The reported error is the discrete maximum-norm
error at $t=T$.  For two successive meshes with sizes $h_1$ and
$h_2$, the observed order is computed by
\[
p=
\frac{\log(E_{h_1}/E_{h_2})}
     {\log(h_1/h_2)}.
\]
\end{revision}

For $T=7$, Figures~\ref{fig:parabolic-u0-comparison} and
\ref{fig:parabolic-u1-comparison} compare the numerical approximations
of $u_0$ and $u_1$ with their exact solutions.

\begin{figure}[H]
\centering
\includegraphics[width=0.5\linewidth]{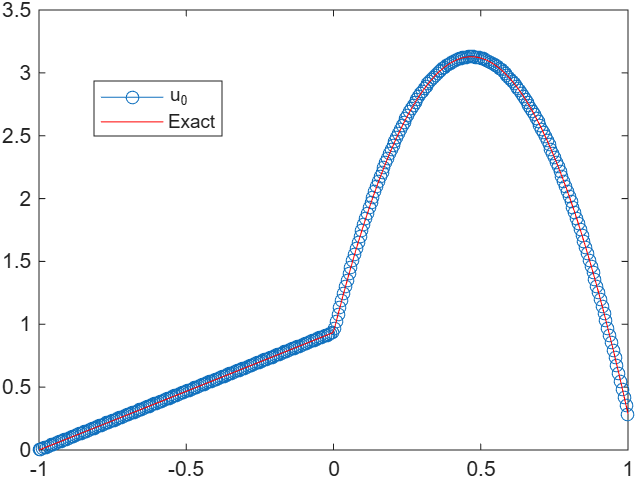}
\caption{Comparison between the numerical and exact solutions for
$u_0$ at $T=7$.}
\label{fig:parabolic-u0-comparison}
\end{figure}

\begin{figure}[H]
\centering
\includegraphics[width=0.5\linewidth]{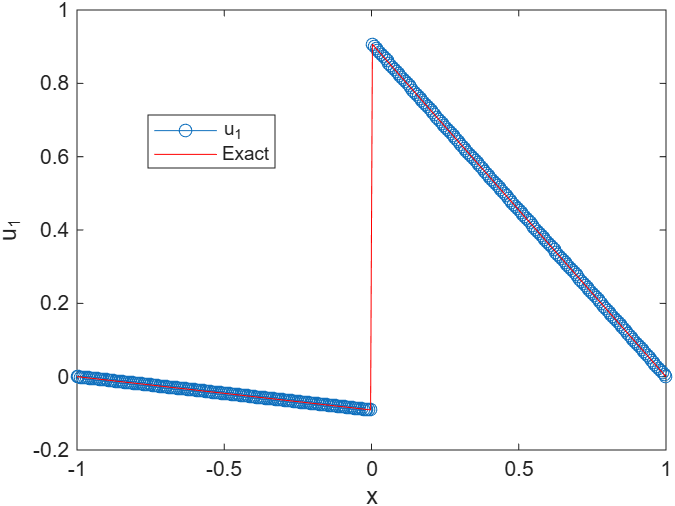}
\caption{Comparison between the numerical and exact solutions for
$u_1$ at $T=7$.}
\label{fig:parabolic-u1-comparison}
\end{figure}

After computing $u_0$, $u_1$, and $s(t)$, the nonlinear solution is
reconstructed as
\[
u(x,t)=u_0(x,t)+s(t)u_1(x,t).
\]
Figure~\ref{fig:parabolic-solution-comparison} compares the
reconstructed solution with the exact nonlinear solution.

\begin{figure}[H]
\centering
\includegraphics[width=0.5\linewidth]{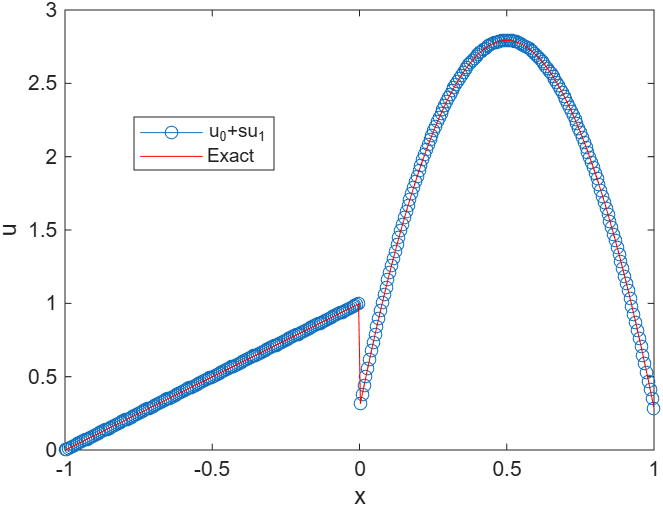}
\caption{Comparison between the reconstructed numerical solution and
the exact nonlinear solution at $T=7$.}
\label{fig:parabolic-solution-comparison}
\end{figure}

\begin{revision}
Table~\ref{tab:parabolic_infinity_error} reports the errors at $T=7$.
The observed orders approach two very closely.  In this manufactured
experiment, the scalar parameter $s$ is recovered to essentially
machine precision from the exact auxiliary traces.  Thus the reported
error primarily reflects the Crank--Nicolson immersed interface
discretization of the auxiliary problems.
\end{revision}

\begin{table}[H]
\centering
\caption{$L^\infty$-errors and observed orders for the nonlinear
parabolic problem at $T=7$, with $\beta^-=1$ and $\beta^+=0.1$.}
\label{tab:parabolic_infinity_error}
\vskip 0.1in
\begin{tabular}{|c|c|c|}
\hline
Mesh size $h$ & $L^\infty$-error & Observed order \\
\hline
0.048780  & $1.1890\times10^{-4}$ & -- \\
\hline
0.024691  & $3.0603\times10^{-5}$ & 1.993 \\
\hline
0.012422  & $7.7658\times10^{-6}$ & 1.996 \\
\hline
0.0062305 & $1.9564\times10^{-6}$ & 1.998 \\
\hline
\end{tabular}
\end{table}

\begin{revision}
Since the transient component of the solution has largely decayed by
$T=7$, we also repeated the computation at the shorter final time
$T=0.1$.  In addition, to examine the effect of a large diffusion
contrast, we considered
\[
\beta^-=1,
\qquad
\beta^+=250,
\]
for which $\beta^+/\beta^-=250$.
Table~\ref{tab:parabolic_short_contrast} reports the results at
$T=0.1$ for both coefficient choices.  In each case the observed
convergence rate approaches two.
\end{revision}

\begin{table}[H]
\centering
\caption{$L^\infty$-errors and observed orders at $T=0.1$ for two
diffusion contrasts.}
\label{tab:parabolic_short_contrast}
\vskip 0.1in
\begin{tabular}{|c|cc|cc|}
\hline
& \multicolumn{2}{c|}{$\beta^-=1,\ \beta^+=0.1$}
& \multicolumn{2}{c|}{$\beta^-=1,\ \beta^+=250$} \\
\hline
$h$ & Error & Order & Error & Order \\
\hline
0.048780
& $1.7583\times10^{-3}$ & --
& $1.0991\times10^{-3}$ & -- \\
\hline
0.024691
& $4.6349\times10^{-4}$ & 1.958
& $2.8515\times10^{-4}$ & 1.982 \\
\hline
0.012422
& $1.1919\times10^{-4}$ & 1.977
& $7.2885\times10^{-5}$ & 1.986 \\
\hline
0.0062305
& $3.0220\times10^{-5}$ & 1.989
& $1.8437\times10^{-5}$ & 1.992 \\
\hline
\end{tabular}
\end{table}

\begin{revision}
Figure~\ref{fig:parabolic-T01} compares the reconstructed numerical
solution with the exact solution at the shorter final time $T=0.1$.
At this time the transient component remains significant, and close
agreement is observed on both sides of the interface.
\end{revision}
\begin{figure}[H]
\centering
\includegraphics[width=0.58\linewidth]{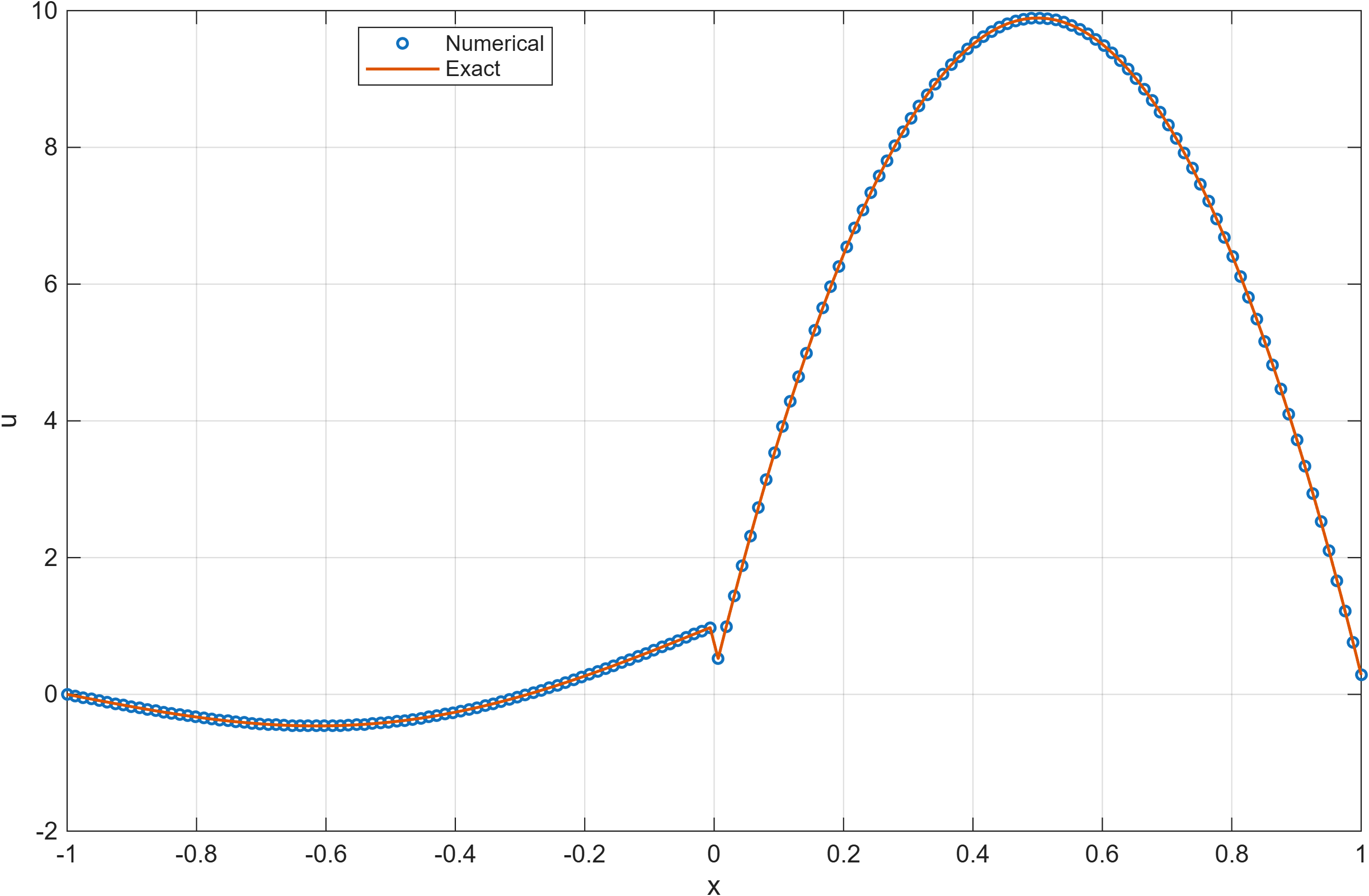}
\caption{Comparison between the reconstructed numerical solution and
the exact nonlinear solution at $T=0.1$, with
$\beta^-=1$ and $\beta^+=0.1$.}
\label{fig:parabolic-T01}
\end{figure}

\begin{revision}
The numerical experiments illustrate two distinct features of the
proposed formulations.  The elliptic example demonstrates the
behavior of the IIM--Newton method in the presence of multiple
solutions, while the parabolic example verifies the second-order
behavior of the Crank--Nicolson IIM combined with the $s$-parameter
reduction.  The additional computations show that the observed
second-order convergence persists at short times and for diffusion
contrasts as large as $\beta^+/\beta^-=250$.
\end{revision}

\section{Remarks on the Newton--IIM Formulation}
\label{sec:newton-iim-not-implementable}

The nonlinear difficulty in \({\bf NPIP}\) is localized at the interface
condition
\[
[u]_\alpha=\lambda u^+u^-.
\]
Since the differential equation is linear away from the interface, it is
natural to ask whether one may first linearize the nonlinear interface
problem by Newton's method at the continuous level and then discretize
the resulting correction equation by the immersed interface method.
We refer to this viewpoint as the Newton--IIM approach.

Let
\[
\mathcal P u:=u_t-(\beta u_x)_x+qu.
\]
Formally, Newton's method seeks a correction \(v\) satisfying
\[
F'(u)v=-F(u),
\]
where the nonlinear operator \(F(u)\) contains both the differential
equation and the nonlinear interface condition.  The corresponding
linearized interface condition takes the form
\[
[v]_\alpha
-
\lambda(u^+v^-+u^-v^+)
=
-[u]_\alpha+\lambda u^+u^-.
\]
Thus the Newton correction problem becomes a linear parabolic interface
problem whose coefficients and interface data depend on the current
iterate \(u\).

At first sight, this appears compatible with the immersed interface
framework.  However, the right-hand side of the correction equation
contains the residual
\[
-(\mathcal Pu-f),
\]
which involves derivatives of the current iterate \(u\).
In practice, the iterate is available only through discrete grid values,
and near the interface these derivatives are influenced by the jump
conditions.  Consequently, evaluating the residual consistently would
require additional differentiation or extrapolation of the numerical
iterate near the interface.

For this reason, the continuous Newton--IIM formulation is less natural
within the standard immersed interface framework than the IIM--Newton
approach considered in the previous section, where the immersed
interface discretization is constructed first and Newton's method is
applied directly to the resulting nonlinear algebraic system.

\bibliographystyle{siam}
\bibliography{references}
 	
\end{document}